\documentclass[12pt]{article}

\usepackage{amssymb}
\usepackage{amsthm}
\usepackage{amsmath}
\usepackage{graphicx}
\usepackage{fullpage}
\usepackage{color}
 \numberwithin{equation}{section}

\theoremstyle{plain}
\newtheorem{thm}{Theorem}[section]
\newtheorem{cor}[thm]{Corollary}

\newtheorem{lem}[thm]{Lemma} 
\newtheorem{prop}[thm]{Proposition}

\theoremstyle{definition}
\newtheorem{defn}[thm]{Definition}

\theoremstyle{remark}

\newtheorem{rem}[thm]{Remark}

\newcommand{\N}{\mathbb{N}}
\newcommand{\R}{\mathbb{R}}

\newcommand{\E}{\mathbb{E}}

\newcommand{\I}{\infty}
\newcommand{\tr}{\text{tr}}
\newcommand{\diag}{\text{diag}}
\newcommand{\bp}{\begin{proof}[\ensuremath{\mathbf{Proof}}]}
\newcommand{\bs}{\begin{proof}[\ensuremath{\mathbf{Solution}}]}
\newcommand{\ep}{\end{proof}}

\begin{document}


\title{Option pricing in the large risk aversion, small transaction cost limit }

\author{Ryan Hynd\footnote{Department of Mathematics, University of Pennsylvania.  Partially supported by NSF grants DMS-1004733 and DMS-1301628.}}

\maketitle 
\begin{abstract}
We characterize the price of a European option on several assets for a very risk averse seller, in a market with small transaction costs as a solution of a nonlinear diffusion equation.  This problem turns out to be one of asymptotic analysis of parabolic PDE, and the interesting feature is the role of a nonlinear PDE eigenvalue problem.  In particular, we generalize previous work of Guy Barles and H. Mete Soner who studied this problem for a European option on a single asset. 

\end{abstract}

%

\section{Introduction}
In a seminal paper \cite{DPZ}, Davis, Panas, and Zariphopoulou presented a model for pricing European options in the presence of transaction costs.  Within this model, Barles and Soner discovered that in markets with small, proportional transaction costs $\approx \sqrt{\epsilon}$, the asking price of a European option on a single asset by a very risk averse $\approx \frac{1}{\epsilon}$ seller is approximately characterized as a solution of a nonlinear Black-Scholes type equation \cite{BS}.  In this work, we extend the result of Barles and Soner and characterize the large risk aversion, small transaction price of European options on several assets. 
\par  In particular, we will show that the asking price of a European option with payoff function $g$, for a very risk averse seller is approximately given as a solution of the PDE
\begin{equation}\label{zEq}
\max_{1\le i\le n}\left\{-z_t -\frac{1}{2}\tr\left(d(p)\sigma\sigma^t d(p)\left(D^2_p z + \frac{1}{\epsilon}(D_p z-y)\otimes (D_p z-y) \right)\right), |z_{y_i}|-\sqrt{\epsilon} p_i\right\}=0,
\end{equation}
where $(t,p,y)\in (0,T)\times (0,\infty)^n\times\R^n$, that satisfies the terminal condition
\begin{equation}\label{TermCondZ}
z(T,p,y)=g(p).
\end{equation}
Here $(D_p z-y)\otimes (D_p z-y)$ is the $n\times n$ matrix with $i,j$th entry $(z_{p_i}-y_i)(z_{p_j}-y_j)$, $d(p):=\diag(p_1,p_2,\dots,p_n)$, $\epsilon $ and $T$ are positive numbers and $\sigma$ is a nonsingular $n\times n$ matrix.  Our goal is to understand the behavior of solutions when $\epsilon$ tends to $0$.
\par  In analogy with the work of Barles and Soner \cite{BS}, we shall see that 
\begin{equation}\label{SevVarZasym}
z^\epsilon(t,p,y)\approx \psi(t,p)+\epsilon u\left(d(p)\frac{D\psi(t,p)-y}{\sqrt{\epsilon}}\right),
\end{equation}
as $\epsilon$ tends to $0$, where $\psi$ is a solution of the {\it nonlinear Black-Scholes equation}
\begin{equation}\label{NonBSpsi}
\begin{cases}
 \psi_t+\lambda\left(d(p)D^2\psi d(p)\right)=0, \quad (t,p)\in (0,T)\times (0,\infty)^n\\
 \hspace{1.4in} \psi = g, \quad (t,p)\in \{T\}\times (0,\infty)^n
 \end{cases}.
 \end{equation}
Note that we have studied this option pricing problem in the case of zero interest rates; see Remark \ref{CaseRpos} for the analog of \eqref{NonBSpsi} with a positive interest rate parameter.

\par We will also see that the nonlinearity $\lambda$ and the function $u$ arising in the error term for $z^\epsilon$ together satisfy the following PDE eigenvalue problem: 
for each $A\in {\mathcal S}(n)$,  find a unique $\lambda=\lambda(A)\in \R$ and a function $x\mapsto u=u(x; A)$ satisfying
\begin{equation}\label{CorrectorPDE}
\max_{1\le i\le n}\left\{\lambda+G(D^2u,Du,x;A), |u_{x_i}|-1\right\}=0, \quad x\in \R^n.
\end{equation}
Here
$$
G(X,p,x;A):=- \frac{1}{2}\tr\sigma\sigma^t\left(A + AXA + \left(x+Ap\right)\otimes\left(x+Ap\right) \right)
$$
for $(X,p,x)\in {\cal S}(n)\times\R^n\times\R^n$, and  ${\cal S}(n)$ denotes the set of real symmetric, $n\times n$ matrices.  The main novelty of this work is our treatment of the eigenvalue problem.  Barles and Soner observed that when $n=1$ equation \eqref{NonBSpsi} reduces to an ODE free boundary problem which has a near explicit solution. This is far from the case in the several asset setting $(n\ge 2)$. Nevertheless we can use PDE techniques to solve this problem. 
\begin{thm}\label{firstmainthm} For each $A\in {\mathcal S}(n)$, there is a unique $\lambda=\lambda(A)\in \R$ such that \eqref{CorrectorPDE} has a viscosity solution 
$u\in C(\R^n)$ satisfying 
\begin{equation}\label{Newugrowth}
\lim_{|x|\rightarrow \infty}\frac{u(x)}{\sum^n_{i=1}|x_i|}=1.
\end{equation}
Moreover, associated to $\lambda(A)$ is a convex solution $u$ satisfying \eqref{Newugrowth}; when $\det A\neq 0$, $u\in C^{1,\alpha}_\text{loc}(\R^n)$ for each $0<\alpha<1$.
\end{thm}
It follows from Theorem \ref{firstmainthm} that the eigenvalue problem associated to the PDE \eqref{CorrectorPDE} has a well defined solution $\lambda: {\mathcal S}(n)\rightarrow \R$.  In order to properly interpret solutions of \eqref{NonBSpsi}, we will need to know that $\lambda$ is a continuous, nondecreasing function with respect to the partial ordering on ${\cal S}(n)$. In fact, we establish this and several other properties in Theorem \ref{lampropthm}.  Then we establish the following theorem, which is the main result of this paper. 
\begin{thm}\label{lastmainthm} Assume that $g\in C((0,\infty)^n)$. Then for each $\epsilon,T>0$ there is a viscosity solution $z^\epsilon\in C((0,T]\times (0,\infty)^n\times\R^n)$  of \eqref{zEq} that satisfies \eqref{TermCondZ}.
Further suppose there is a constant $L$ for which  
\begin{equation}\label{gEstimate}
0\le g(p)\le L
\end{equation}
or 
\begin{equation}\label{gEstimate2}
0\le g(p)\le L\sum^n_{i=1}p_i\quad \text{and}\quad \lim_{|p|\rightarrow \infty}\frac{g(p)}{\sum^n_{i=1}p_i}=L.
\end{equation}
Then, as $\epsilon$ tends to $0$, $z^\epsilon$ converges uniformly on compact subsets of $(0,T) \times (0,\infty)^n\times\R^n$ to a viscosity solution of equation \eqref{NonBSpsi}.
\end{thm}

In section \ref{AnalEig}, we study the eigenvalue problem in detail and prove Theorem \ref{firstmainthm}.  In section \ref{EigProp}, we prove Theorem \ref{lampropthm} which verifies some important properties of $\lambda$.  Finally in section \ref{zConverge}, we establish Theorem \ref{lastmainthm}, which characterizes the large risk aversion, small transaction cost option price.   Before undertaking this work, we present the mathematical model from which the equations derive  and perform some formal computations that will guide our intuition for analyzing $z^\epsilon$ for $\epsilon$ small.

{\bf The market model}. Following the work of \cite{DPZ, BS}, we consider a Brownian motion based financial market consisting of $n$ stocks and a money market account (a ``bond") with interest rate $r\ge 0$. The stocks are modeled as a stochastic process satisfying the SDE

\begin{equation}\label{stockSDE}
dP^i(s) =\sum^n_{j=1}\sigma_{ij} P^i(s)dW^j(s), \quad s\ge 0, \quad i=1,\dots, n \nonumber
\end{equation}
where $(W(t), t\ge 0)$ is a standard $n$-dimensional Brownian motion.  
We assume each participant in the market assumes a trading strategy, which is simply a way of purchasing and selling shares of stock and 
the money market account.  Furthermore, in this model we assume that participants pay transaction costs that are proportional to the 
amount of the underlying stock; the proportionality constant we use is $\sqrt{\epsilon}$.

\par On a time interval $[t,T]$, a trading strategy will be modeled by a pair of vector processes
$(L, M)=((L^1, \dots, L^n),(M^1, \dots, M^n))$.  Here $L^i(s)$ represents 
the cumulative purchases of the $i$th stock and $M^i(s)$ represents the cumulative sales of the $i$th stock at time $s\in [t,T]$; we assume $L^i, M^i$ are 
non-decreasing processes, adapted to the filtration generated by $W$, that satisfy  $L^i(t)=M^i(t)=0$ for $i=1,\dots,n$. Associated to a given trading strategy $(L,M)$ is
a process $X$, the amount of dollars held in the money market, and processes $Y^i$, the number of shares of the $i$th stock held, for $i=1,\dots,n.$ These processes
are modeled by the SDE
$$
\begin{cases}
dX(s) = r X(s)ds + \sum^n_{i=1}\left( - (1+\sqrt{\epsilon})P^i(s) dL^i(s)+(1-\sqrt{\epsilon})P^i(s)dM^i(s)\right)\\
dY^i(s) = dL^i(s)- dM^i(s) \quad i=1,\dots, n 
\end{cases}\;t\le s\le T.
$$

\par 
We assume that for a given amount of wealth $w\in \R$, a seller of a European option with maturity $T$ and payoff $g(P(T))\ge 0 $ has the {\it utility}

\begin{equation}
U_\epsilon(w)=1-e^{-w/\epsilon}. \nonumber
\end{equation}
In particular, the seller has constant {\it risk aversion} 
$$\frac{-U_ \epsilon''(w)}{U_ \epsilon'(w)}=\frac{1}{\epsilon}.
$$
If the seller does not sell the option, his expected utility from final wealth is
\begin{equation}\label{Veq}
v^{\epsilon,f}(t,x,y,p)=\sup_{L,M}\E U_\epsilon(X(T)+Y(T)\cdot P(T));  \nonumber
\end{equation}
here $X(t)=x, Y(t)=y$ and $P(t)=p$.  If he does sell the option, he will have to payout $g(P(T))$ at time $T$, so his expected utility from final wealth is

\begin{equation}
v^\epsilon(t,x,y,p)=\sup_{L,M}\E U_\epsilon(X(T)+Y(T)\cdot P(T) - g(P(T))). \nonumber
\end{equation}
Since $U_\epsilon$ is monotone increasing, $v^{\epsilon}\le v^{\epsilon,f}$.  We define the seller's price $\Lambda_\epsilon$ as the 
amount which offsets this difference (and makes the seller ``indifferent" to selling the option or not)
\begin{equation}
v^\epsilon(t,x+\Lambda_\epsilon,y,p)=v^{\epsilon,f}(t,x,y,p). \nonumber
\end{equation}
See \cite{RC} for more on this approach to option pricing. 

\par We now specialize to the case $r=0$. This is done without any loss of generality as
\begin{equation}\label{vrequal0}
(t,x,y,p)\mapsto v(t,e^{r(T-t)}x,y,e^{r(T-t)}p)
\end{equation}
is the value function for $r>0$, provided $v$ is the value function when $r=0$.  As in the single asset case \cite{DPZ}, we have the following proposition. Part $(i)$ follows directly from Theorem 2 and Theorem 3 of \cite{DPZ}; part $(ii)$
follows from basic calculus.

\begin{prop}\label{zSolnProp}
(i) $v^\epsilon, v^{\epsilon,f}$ are continuous viscosity solutions of the PDE
{\small
\begin{equation}\label{veq}
\max_{1\le i\le n}\left\{v_{y_i}-(1+\sqrt{\epsilon})p_i v_x,-v_{y_i}+(1-\sqrt{\epsilon})p_i v_x, v_t + \frac{1}{2}d(p)\sigma\sigma^t d(p)\cdot D^2_p v\right\}=0,
\end{equation}
}
for $(t,y,p)\in (0,T)\times\R^n\times (0,\I)^n$ and satisfy 
\begin{equation}\label{vterm}
v^\epsilon(T,x,y,p)=1-\exp( -(x+y\cdot p -g(p))/\epsilon ) \quad\text{and} \quad v^{\epsilon,f}(T,x,y,p)=1-\exp(-(x+y\cdot p)/\epsilon). \nonumber
\end{equation}
(ii)Define $z^\epsilon, z^{\epsilon,f}$ implicitly via
$$
\begin{cases}
v^\epsilon=U_\epsilon(x+y\cdot p - z^\epsilon)=1- \exp(-(x+y\cdot p - z^\epsilon)/\epsilon) \nonumber\\
v^{\epsilon,f}=U_\epsilon(x+y\cdot p - z^{\epsilon,f})=1-  \exp(-(x+y\cdot p - z^{\epsilon,f})/\epsilon)
\end{cases}.
$$
Then $z^\epsilon, z^{\epsilon,f}$ are viscosity solutions of \eqref{zEq} satisfying the terminal conditions
$$
z^{\epsilon}(T,p,y)=g(p)\quad \text{and}\quad z^{\epsilon,f}(T,p,y)=0.
$$ 
\end{prop}

\begin{rem}The main virtue of working with the exponential utility function is that the value functions typically depend on the $x$ variable in a simple way. Notice that 
$$
X(s)=x + \int^{s}_{t}\left\{-(1+\sqrt{\epsilon})P(s)\cdot dL(s) + (1- \sqrt{\epsilon})P(s)\cdot dM(s)\right\}, \quad t\le s \le T
$$
and so $v=v^\epsilon,v^{\epsilon,f}$ satisfy
\begin{equation}\label{vVarReduction}
v(t,x,y,p) =1+ e^{-x/\epsilon}(v(t,0,y,p)-1).
\end{equation}
This is convenient as it reduces the variable dependence of solutions of \eqref{veq}.  Moreover, using \eqref{vVarReduction}, it is straightforward to check that 
$$
z^\epsilon, z^{\epsilon,f}\; \text{are {\it independent }of $x$}. 
$$
\end{rem}
\noindent {\bf The large risk aversion, small transaction cost limit.} 
Directly from the definitions of $\Lambda_\epsilon$, $z^\epsilon$ and $z^{\epsilon,f}$ we see
\begin{equation}
\Lambda_\epsilon=z^\epsilon - z^{\epsilon,f}.\nonumber
\end{equation}
Consequently, in order to understand the limiting option price it suffices to study $\lim_{\epsilon\rightarrow 0^+}z^\epsilon$ and $\lim_{\epsilon\rightarrow 0^+}z^{\epsilon,f}$. Therefore, the problem of characterizing the limiting option price is 
reduced to that of asymptotic analysis of solutions nonlinear parabolic PDE.

\par Below, we give a step-by-step formal derivation of how we arrived at the PDE [equation \eqref{NonBSpsi}] for the limit $\psi$ and the PDE [equation \eqref{CorrectorPDE}] arising in the eigenvalue problem.
These heuristic calculations are arguably the most important part of our work since the techniques we use later are founded on these results.  These computations are based largely on section 3.2 of \cite{BS}. 
\\\\
{\bf Step 1.} $\max_{1\le i\le n}\{|z_{y_i}^\epsilon|-\sqrt{\epsilon}p_i\}\le 0$, so we expect $\lim_{\epsilon\rightarrow 0^+}z^\epsilon$ to be independent of $y$.  This observation leads to the choice of {\it ansatz}
$$
z^{\epsilon}(t,p,y)\approx \psi(t,p)+\epsilon u(x^\epsilon(t,p,y)),
$$
for $\epsilon$ small. Here $\psi$, $u$ and $x^\epsilon$ are yet to be determined. We also observe that since
$$
\epsilon |x_{y_i}^\epsilon\cdot Du(x^\epsilon)|\approx |z_{y_i}^\epsilon|\le \sqrt{\epsilon}p_i,
$$
$x^\epsilon$ (and its derivatives) should probably scale at worst like $1/\sqrt{\epsilon}$.  
With this assumption, 
\begin{eqnarray}
I^\epsilon&:=&-z^\epsilon_t -\frac{1}{2}\tr\left(d(p)\sigma\sigma^t d(p)\left(D^2_p z^\epsilon + \frac{1}{\epsilon}(D_p z^\epsilon-y)\otimes (D_p z^\epsilon-y) \right)\right)\nonumber \\
 &\approx&  -\psi_t
-\frac{1}{2}\tr\sigma\sigma^t\text{{\huge(}}d(p)D^2\psi d(p)+ \left(\sqrt{\epsilon}D_px^\epsilon d(p))^tD^2u(x^\epsilon)(\sqrt{\epsilon}D_px^\epsilon d(p)\right) + \nonumber \\
&& \left. \left(d(p)\frac{D\psi - y}{\sqrt{\epsilon}} +(\sqrt{\epsilon}D_px^\epsilon d(p))^tDu(x^\epsilon)\right)\otimes \left(d(p)\frac{D\psi - y}{\sqrt{\epsilon}} +(\sqrt{\epsilon}D_px^\epsilon d(p))^tDu(x^\epsilon)\right)\right).
\nonumber
\end{eqnarray}

\noindent {\bf Step 2.} Notice that 
\begin{align*}
\sqrt{\epsilon}D_p\left(d(p)\frac{D\psi - y}{\sqrt{\epsilon}}\right)d(p) &= d(p)D^2\psi d(p) 
+\sqrt{\epsilon}\; \text{diag}\left(d(p)\frac{D\psi - y}{\sqrt{\epsilon}}\right)
\end{align*}
This minor observation and the above computations lead us to choose the new ``variable"
\begin{equation}
x^\epsilon:=d(p)\frac{D\psi-y}{\sqrt{\epsilon}} \nonumber
\end{equation}
and the new ``parameter" 
\begin{equation}
A:=d(p)D^2\psi d(p). \nonumber
\end{equation}
We further {\it postulate} that there is a function $\lambda$ such that
\begin{equation}\label{AlmostPDE}
\psi_t+\lambda(A)=0. \nonumber
\end{equation}

\noindent {\bf Step 3.} With the above choices and postulate,

\begin{equation}
I^\epsilon\approx\lambda(A)+G(D^2u(x^\epsilon),Du(x^\epsilon),x^\epsilon;A) \nonumber
\end{equation} 
and also for $i=1,\dots,n$
\begin{equation}
|u_{x_i}(x^\epsilon)| \lesssim 1. \nonumber
\end{equation}
Since 
$$
\max_{1\le i\le n}\{I^\epsilon,|z_{y_i}|-\sqrt{\epsilon}p_i\}=0,
$$
we will require that $u$ and $\lambda(A)$ satisfy the PDE \eqref{CorrectorPDE}. In view of estimates we will later derive on $z^\epsilon$ (see inequality \eqref{zgrowth1}), we additionally require $u$ to satisfy the growth condition \eqref{Newugrowth}.
In summary, we are lead to eigenvalue problem outlined in \eqref{CorrectorPDE} and \eqref{Newugrowth}.

\par If we can solve the eigenvalue problem \eqref{CorrectorPDE} {\it uniquely, for a nondecreasing function} $\lambda$, we have the solution of the PDE \eqref{NonBSpsi}  
as a candidate for the limit of $z^\epsilon$ as $\epsilon\rightarrow 0^+.$  We remark that the procedure described above is philosophically similar
to the formal asymptotics of periodic homogenization. In analogy with that framework, $\lambda$ plays the role of the effective Hamiltonian, and the eigenvalue problem 
plays the role of the cell problem \cite{LPV, LCE}. Finally, we note that the same heuristic argument shows that $z^{\epsilon,f}$ satisfies the PDE \eqref{NonBSpsi} except with the terminal condition $\psi|_{t=T}\equiv 0$. 
So we conclude  $\lim_{\epsilon\rightarrow 0^+}z^{\epsilon,f}=0$ and in particular, 
$$
\lim_{\epsilon\rightarrow 0^+}\Lambda_\epsilon=\lim_{\epsilon\rightarrow 0^+}z^\epsilon.
$$
\begin{rem}\label{CaseRpos}
It is possible to use the change of variables in \eqref{vrequal0} to deduce that the corresponding nonlinear Black-Scholes equation for a positive interest rate $r>0$ is
$$
\psi_t+e^{-r(T-t)}\lambda\left(e^{r(T-t)}d(p)D^2\psi d(p)\right)+ rp\cdot D\psi-r\psi=0.
$$
See section 3.1 of \cite{BS} for more details.
\end{rem}
\begin{rem}
While Theorem \ref{lastmainthm} does not cover every possible payoff function $g$ for a European option on several assets, it covers many that arise in practice. 
Example payoffs which have financial interpretations are
$$
g(p)=\left(\sum^n_{i=1}p_i - K\right)^+, \quad g(p)=\sum^n_{i=1}\left(K- p_i \right)^+, \quad g(p)=K- \left(\prod^n_{i=1}p_i\right)^{1/n}.
$$
\end{rem}

\begin{rem}
In view of the asymptotic expansion \eqref{SevVarZasym} and the growth condition \eqref{Newugrowth}, 
$$
\Lambda^\epsilon(t,p,y)\approx \psi(t,p)+\sum^n_{i=1}\sqrt{\epsilon}p_i|\psi_{p_i}(t,p)-y_i|
$$
as $\epsilon\rightarrow 0^+.$  Thus for small $\epsilon$, $\Lambda_\epsilon$ is a sum of the limiting option price plus a term that naturally resembles a transaction cost. 
\end{rem}

\section{A nonlinear eigenvalue problem}\label{AnalEig}

In this section, we prove the first part of Theorem \ref{firstmainthm} which is the statement that the eigenvalue problem is well posed. The methods are largely based on the approach given in our previous article \cite{Hynd2}, however we consider this work a considerable extension.  First, we give a definition that will allow for clear statements to follow.  This definition involves viscosity solutions of nonlinear elliptic PDE and we refer readers to the standard sources for background material on this concept \cite{CIL, FS, BC}. We shall also employ the notation of \cite{CIL}, and for any $n\times n$ matrix $B$ below, we will denote 
$|B|:=\sup\{|Bw|: w\in \R^n, |w|=1\}$.

\begin{defn}\label{ViscEigDef}
$u\in USC(\R^n)$ is a {\it viscosity subsolution} of $\eqref{CorrectorPDE}$ with {\it eigenvalue $\lambda\in\R$}
if for each $x_0\in\R^n$, 
$$
\max_{1\le i\le n}\left\{\lambda +G(D^2\phi(x_0),D\phi(x_0),x_0;A), |\varphi_{x_i}(x_0)|-1\right\}\le 0,
$$
whenever $u-\varphi$ has a local maximum at $x_0$ and $\varphi \in C^2(\R^n)$. $v\in LSC(\R^n)$ is a {\it viscosity supersolution}
of $\eqref{CorrectorPDE}$ with {\it eigenvalue $\mu\in\R$} if for each $y_0\in\R^n$, 
$$
\max_{1\le i\le n}\left\{\mu+G(D^2\psi(y_0),D\psi(y_0),y_0;A), |\psi_{x_i}(y_0) |-1\right\}\ge 0,
$$
whenever $v-\psi$ has a local minimum at $y_0$ and $\psi \in C^2(\R^n)$.  $u\in C(\R^n)$ is a {\it viscosity 
solution} of \eqref{CorrectorPDE} with {\it eigenvalue} $\lambda\in\R$ if it is both a viscosity sub- and supersolution of \eqref{CorrectorPDE} with
eigenvalue $\lambda.$
\end{defn}

\subsection{Comparison of eigenvalues}
We start our treatment of the eigenvalue problem by establishing a fundamental comparison principle that will
allow us to compare eigenvalues associated with sub- and supersolutions of \eqref{CorrectorPDE}. 

\begin{prop}\label{Newcomparison}
Suppose $u$ is a subsolution of \eqref{CorrectorPDE} with eigenvalue $\lambda$ and that 
 $v$ is a supersolution of \eqref{CorrectorPDE} with eigenvalue $\mu$. If in addition
\begin{equation}\label{NewLimitForWW}
\limsup_{|x|\rightarrow \infty}\frac{u(x)}{\sum^n_{i=1}|x_i|}\le 1\le \liminf_{|x|\rightarrow \infty}\frac{v(x)}{{\sum^n_{i=1}|x_i|}}, \nonumber
\end{equation}
then $\lambda\le \mu.$
\end{prop}

\begin{proof}   
\par 1. Fix $0<\tau <1$ and set 
$$
w^\tau(x,y)=\tau u(x) - v(y), \quad x,y\in\R^n.
$$
For $\delta >0$, we also set 
$$
\varphi_\delta(x,y)=\frac{1}{2\delta}|x-y|^2, \quad x,y\in\R^n.
$$
The inequality 
\begin{eqnarray}\label{wepsIneq}
w^\tau(x,y)-\varphi_\delta(x,y) & = &\tau (u(x)-u(y)) - \frac{1}{2\delta}|x-y|^2 + \tau u(y)-v(y) \nonumber \\
                                                              &\le &\left(\sqrt{n}|x-y| - \frac{1}{2\delta}|x-y|^2\right) +  \tau u(y)-v(y) \nonumber
\end{eqnarray}
implies 
$$
\lim_{|x|+|y|\rightarrow \infty}\left\{w^\tau(x,y)-\varphi_\delta(x,y) \right\} = -\infty.
$$
Therefore, $w^\tau-\varphi_\delta$ achieves a global maximum at a point $(x_\delta, y_\delta)\in \R^n\times\R^n$.  

\par 2. According to the Crandall-Ishii Lemma \cite{CI} (see in particular Theorem 3.2 in \cite{CIL}), for each $\rho>0$, there are $X,Y\in {\mathcal S}(n)$ such that 

$$
\left(\frac{x_\delta - y_\delta}{\delta}, X\right)=\left(D_x\varphi_\delta(x_\delta,y_\delta), X\right)\in \overline{J}^{2,+}(\tau u)(x_\delta),
$$ 
\begin{equation}\label{vCalcCond}
\left(\frac{x_\delta - y_\delta}{\delta}, Y\right)=\left(-D_y\varphi_\delta(x_\delta,y_\delta), Y\right)\in \overline{J}^{2,-}v(y_\delta),
\end{equation}
and 
\begin{equation}\label{ImportantMatIneq}
\left(\begin{array}{cc}
X & 0 \\
0 & - Y
\end{array}\right)\le A +\rho A^2.
\end{equation}
Here 
$$
A=D^2\varphi_\delta(x_\delta,y_\delta)=\frac{1}{\delta}
\left(\begin{array}{cc}
I_n & -I_n \\
-I_n & I_n
\end{array}\right).
$$
Notice that \eqref{ImportantMatIneq} implies that $X\le Y.$

\par 3. Set 
$$
p_\delta=\frac{1}{\tau}\frac{x_\delta-y_\delta}{\delta}
$$
and note that $p_\delta\in \overline{J}^{1,+}u(x_\delta)$. Also observe that as $\max_{1\le i\le n}|u_{x_i}|\le 1$ (in the sense of viscosity solutions),
$$
\max_{1\le i\le n}|\tau p_\delta\cdot e_i|\le \tau <1. 
$$
Here $\{e_1,\dots,e_n\}$ denotes the usual standard basis in $\R^n.$ Since $v$ is a viscosity supersolution of \eqref{CorrectorPDE} with eigenvalue $\mu$, we have
$$
\mu + G(Y,p_\delta,y_\delta;A)\ge 0
$$
by  \eqref{vCalcCond}. As $u$ is a viscosity subsolution of \eqref{CorrectorPDE} with eigenvalue $\lambda$,
$$
\lambda + G(X/\tau,p_\delta,x_\delta;A)\le 0.
$$
Therefore,
\begin{eqnarray}\label{lammuest}
\tau \lambda -\mu &\le& - \tau  G(X/\tau,p_\delta,x_\delta;A) +G(Y,p_\delta,y_\delta;A)\nonumber \\
&\le & \frac{1}{2}(\tau -1)\tr \sigma\sigma^tA + \frac{1}{2}\tau |\sigma^t(x_\delta+Ap_\delta)|^2 - \frac{1}{2}|\sigma^t(y_\delta+\tau Ap_\delta)|^2.
\end{eqnarray}

\par 4. We now claim that $x_\delta\in\R^n$ is bounded for all small enough $\delta>0.$ If not, then there is a subsequence of $\delta\rightarrow 0$ such that
$(w^{\tau}-\varphi_\delta)(x_\delta,y_\delta)$ tends to $-\infty$.  Indeed 
\begin{align*}
(w^{\tau}-\varphi_\delta)(x_\delta,y_\delta)&=(\tau u(x_\delta)-v(x_\delta) ) + v(x_\delta)-v(y_\delta) -\frac{|x_\delta - y_\delta|^2}{2\delta}\nonumber \\
&\le (\tau u(x_\delta)-v(x_\delta) ) + \sqrt{n}|x_\delta-y_\delta|-\frac{|x_\delta - y_\delta|^2}{2\delta}\nonumber \\ 
&\le   (\tau u(x_\delta)-v(x_\delta) ) +\frac{n\delta}{2} \nonumber
\end{align*}
which tends to $-\infty$ as $\delta \rightarrow 0$ provided $\lim_{\delta\rightarrow 0^+}|x_\delta|=+\infty.$ This would be the case for some sequence of  $\delta\rightarrow 0$
if $x_\delta $ is unbounded.  However,
\begin{equation}
(w^{\tau}-\varphi_\delta)(x_\delta,y_\delta)=\max_{x,y\in\R^n}\left\{\tau u(x) - v(y) -\frac{|x-y|^2}{2\delta}\right\}\nonumber\ge \tau u(0)-v(0)>-\infty
\end{equation}
and thus $x_\delta$ lies in a bounded subset of $\R^n$.  Likewise, $y_\delta\in\R^n$ is bounded. 

\par Since 
$$
\lim_{\delta\rightarrow 0^+}\frac{|x_\delta-y_\delta|^2}{2\delta}\rightarrow 0
$$
(by Lemma 3.1 in \cite{CIL}), the sequence $( (x_\delta,y_\delta) )_{\delta>0}$
has a cluster point $(x_\tau,x_\tau)$ for a sequence of $\delta\rightarrow 0.$ Note also that  $p_\delta\in \R^n$ is a bounded sequence so we can also assume that $p_\delta\rightarrow p$
as $\delta\rightarrow 0$, for some $p\in \R^n$ with $\max_{1\le i\le n}|p_i|\le 1.$  Passing to this limit in \eqref{lammuest} gives
\begin{eqnarray}
\tau \lambda -\mu &\le &\frac{1}{2} (\tau -1)\tr\sigma\sigma^tA  +  \frac{1}{2}\tau |\sigma^t(x_\tau+Ap)|^2 - \frac{1}{2}|\sigma^t(x_\tau+\tau Ap)|^2 \nonumber \\
&\le &  \frac{1}{2}(\tau -1)\tr\sigma\sigma^tA + \frac{1}{2}(\tau -1)|\sigma^tx_\tau|^2 + \frac{1}{2}\tau(1-\tau)|\sigma^tAp|^2\nonumber \\
&\le &  \frac{1}{2}(\tau -1)\tr\sigma\sigma^tA +\frac{n}{2}\tau(1-\tau)|\sigma^tA|^2. \nonumber
\end{eqnarray}
We conclude by letting $\tau\rightarrow 1^-$.
\end{proof}
The following corollary is immediate. 
\begin{cor}
For each $A\in S(n)$, there can be at most one $\lambda$ such that \eqref{CorrectorPDE} has a solution $u$ with eigenvalue $\lambda$ satisfying the growth condition \eqref{Newugrowth}. 
\end{cor}
Now that we know that there can be at most one solution of the eigenvalue problem, we are left to answer the question of whether or not a single solution exists. We shall see that this is in fact the case.
To approximate the values of a potential eigenvalue, we study 
\begin{equation}\label{AdeltaPDE}
\max_{1\le i\le n}\left\{\delta u +G(D^2u,Du,x;A), |u_{x_i}|-1\right\}=0,\; x\in\R^n
\end{equation}
for $\delta>0$ and small, and seek solutions that satisfy growth condition \eqref{Newugrowth}.  The goal is to show that  the above PDE has a unique solution $u_\delta$ and that there is a sequence of $\delta\rightarrow 0^+$ such that
$\delta u_\delta(0)\rightarrow \lambda(A)$.  Moreover, we hope that $u_\delta - u_\delta(0)$ converges to a solution $u$ of \eqref{CorrectorPDE}.
First, we address the question of uniqueness of solutions of \eqref{AdeltaPDE}. As this can be handled similar to the comparison
principle for eigenvalues, we omit the proof.

\begin{prop}\label{AcomparisonU}
Suppose $u$ is a subsolution of \eqref{AdeltaPDE} and that 
 $v$ is a supersolution of \eqref{AdeltaPDE}. If in addition
$$
\limsup_{|x|\rightarrow \infty}\frac{u(x)}{\sum^n_{i=1}|x_i|}\le 1\le \liminf_{|x|\rightarrow \infty}\frac{v(x)}{\sum^n_{i=1}|x_i|},
$$
then $u\le v.$
\end{prop}

\begin{cor}
For each $A\in S(n)$, there can be at most one solution of \eqref{AdeltaPDE} satisfying \eqref{Newugrowth}. 
\end{cor}
To establish existence, we need sub- and supersolutions with the appropriate growth as $|x|\rightarrow \infty$.
\begin{lem} Fix $0<\delta<1$ and $A\in {\cal S}(n)$.\\
$(i)$ There is a constant $K=K(A)>0$ such that  
\begin{equation}\label{Newviscsub}
\underline{u}(x)=\left(\sum^n_{i=1}|x_i|- K\right)^+ +\frac{\tr\sigma^t\sigma A}{2\delta},\; x\in\R^n 
\end{equation}
is a viscosity subsolution of \eqref{AdeltaPDE} satisfying the growth condition \eqref{Newugrowth}. \\
$(ii)$ There is a constant $K=K(A)>0$ such that 

\begin{equation}\label{Newviscsuper}
\overline{u}(x)=\frac{K}{\delta} + \sum^n_{i=1}\begin{cases}\frac{1}{2}x_i^2,\;\;\;\;\;\; |x_i|\le 1\\  |x_i|-\frac{1}{2},|x_i|\ge 1\end{cases} ,\; x\in\R^n
\end{equation}
is a viscosity supersolution of  \eqref{AdeltaPDE} satisfying the growth condition \eqref{Newugrowth}.
\end{lem}

\begin{proof} $(i)$ Choose $K>0$ such that 
$$
\left(\sum^n_{i=1}|x_i|- K\right)^+\le \frac{1}{2}\tr \sigma\sigma^tA+\frac{1}{2}\left(|\sigma^tx| - \sqrt{n}|\sigma^tA|\right)^2, \quad x\in \R^n.
$$
As $\underline{u}$ is convex and as $\max_{1\le i\le n}|\underline{u}_{x_i}|=1$, if $(p,X)\in  J^{2,+}\underline{u}(x_0)$ then
$$
\max_{1\le i\le n}|p_i|\le 1\quad \text{and}\quad X\ge 0.
$$
Hence, 
\begin{eqnarray}
\delta \underline{u}(x_0) + G(X,p,x_0;A)&\le &\left(\sum^n_{i=1}|x_0\cdot e_i|- K\right)^+   -\frac{1}{2}\tr \sigma\sigma^tA  -\frac{1}{2}|\sigma^t(x_0 + Ap)|^2  \nonumber \\
 &\le & \left(\sum^n_{i=1}|x_0\cdot e_i|- K\right)^+ -\frac{1}{2}\tr \sigma\sigma^tA   - \frac{1}{2}\left(|\sigma^tx_0| - \sqrt{n}|\sigma^tA|\right)^2\le 0 \nonumber
\end{eqnarray}
Thus $\underline{u}$ is a viscosity subsolution. 

\par $(ii)$ Select 
$$
K:=\max\left\{-G(I,x,x;A): \; \max_{1\le i\le n}|x_i|\le 1 \right\}
$$
and assume that $(p,X)\in J^{2,-}\underline{u}(x_0)$. If $|x_0\cdot e_i|<1$ for all $i=1,\dots,n$, $\bar{u}$ is smooth in a neighborhood of $x_0$ and 
$$
\begin{cases}
\bar{u}(x_0)=\frac{K}{\delta} +\frac{|x_0|^2}{2}\\
D\bar{u}(x_0)=x_0=p\\
D^2\bar{u}(x_0)=I_n=X
\end{cases}.
$$
Therefore,
{\small
$$
\delta\bar{u}(x_0) + G(X,p,x_0;A) \ge K + G(I,x_0,x_0;A) \ge 0,
$$
}which implies 
\begin{equation}\label{NewbaruSupersoln}
\max_{1\le i\le n}\left\{\delta\bar{u}(x_0) + G(X,p,x_0;A) , |p_i|- 1\right\}\ge 0.
\end{equation}
Now suppose $|x_0\cdot e_i|\ge 1$ for some $i\in \{1,\dots,n\}$. $\bar{u}\in C^{1}(\R^n)$, so $p_i=\bar{u}_{x_i}(x_0)=x_0\cdot e_i/|x_0\cdot e_i|$ and in particular $|p_i|=1$. Thus
\eqref{NewbaruSupersoln} still holds, and consequently,  $\bar{u}$ is a viscosity supersolution.
\end{proof}
As the existence of a unique viscosity solution now follows directly from applying Perron's method (see section 4 of \cite{CIL}, for instance), we omit the proof. 
\begin{thm}\label{yeahUnique}
Let $A\in {\cal S}(n)$. For each $\delta\in (0,1)$, there exists a unique viscosity solution $u_\delta$ of the PDE \eqref{AdeltaPDE} satisfying the growth condition \eqref{Newugrowth}.
\end{thm}

\subsection{Basic estimates}
With the existence of a unique solution $u_\delta$ of  \eqref{AdeltaPDE}, our goal is establish some estimates on $u_\delta$ that will help us pass to the limit as $\delta\rightarrow 0.$ 
A fundamental property of $u_\delta$ that we deduce below is that it is convex.  Other important estimates of $u_\delta$ will be derived 
directly from this.  The method of proof is virtually the same as in \cite{Hynd2} (Lemma 3.7) and originates from \cite{KO}.

\begin{prop}\label{PropConvexUdelta}
Let $u_\delta\in C(\R^n)$ be the unique solution of the PDE \eqref{AdeltaPDE} subject to the growth condition \eqref{Newugrowth}. Then $u_\delta$ is convex. 
\end{prop}

\begin{proof}
1.  We first assume $u_\delta\in C^2(\R^n)$ and for ease of notation, we write $u$ for $u_\delta $. Fix  $0< \tau <1$ and set 
$$
{\mathcal C}^\tau(x,y)=\tau u\left(\frac{x+y}{2}\right)-\frac{u(x) + u(y)}{2}, \quad x,y\in\R^n.
$$
We aim to bound ${\mathcal C}^\tau$ from above and later send $\tau\rightarrow 1^-.$  We claim that ${\cal C}^\tau$ has a maximizing 
point $(x_\tau, y_\tau)\in \R^n\times \R^n$.  It suffices to show 
$$
\lim_{|x|+|y|\rightarrow\infty}{\cal C}^\tau(x,y)=-\infty.
$$

 Let $(x_k, y_k)\in \R^n\times \R^n$ be such that $|x_k|+|y_k|\rightarrow \infty$, as $k\rightarrow \infty$. Observe
 \begin{align*}
\frac{{\cal C}^\tau(x_k,y_k)}{|x_k|+|y_k|}&=\frac{\tau}{2}\frac{u\left(\frac{x_k+y_k}{2}\right)}{|x_k|+|y_k|}-
\frac{1}{2}\left\{\frac{|x_k|}{|x_k|+|y_k|}\frac{u(x_k)}{|x_k|}+\frac{|x_k|}{|x_k|+|y_k|}\frac{u(y_k)}{|y_k|}\right\}\\
&\le\frac{\tau}{2}\frac{u\left(\frac{x_k+y_k}{2}\right)}{\left|\frac{x_k+y_k}{2}\right|}-
\frac{1}{2}\left\{\frac{|x_k|}{|x_k|+|y_k|}\frac{u(x_k)}{|x_k|}+\frac{|x_k|}{|x_k|+|y_k|}\frac{u(y_k)}{|y_k|}\right\},
 \end{align*}
 when of course each numerator above is positive. This manipulation can be used to show  
 $$
 \limsup_{k\rightarrow \infty}\frac{{\cal C}^\tau(x_k,y_k)}{|x_k|+|y_k|}\le \frac{\tau -1}{2}<0. 
 $$
Hence, $\limsup_{k\rightarrow \infty}{\cal C}^\tau(x_k,y_k)=-\infty$. The claim follows as $\{(x_k, y_k)\}_{k\in \N}$ was 
an arbitrary, unbounded sequence. 

\par 2.  At any maximizing point $(x_\tau,y_\tau)$ for ${\cal C}^\tau$,
$$
0=D_x{\mathcal C}^\tau(x_\tau, y_\tau)=\frac{\tau}{2}Du\left(\frac{x_\tau+y_\tau}{2}\right)-\frac{1}{2}Du(x_\tau)
$$
and
$$
0=D_y{\mathcal C}^\tau(x_\tau, y_\tau)=\frac{\tau}{2}Du\left(\frac{x_\tau+y_\tau}{2}\right)-\frac{1}{2}Du(y_\tau).
$$
Thus,
$$
\tau Du\left(\frac{x_\tau+y_\tau}{2}\right)=Du(x_\tau)=Du(y_\tau).
$$
Also observe that $v\mapsto {\mathcal C}^\tau(x_\tau + v, y_\tau + v)$ has a maximum at $v=0$ which implies

$$
0\ge \tau D^2u\left(\frac{x_\tau+y_\tau}{2}\right)- \frac{D^2u(x_\tau)+D^2u(y_\tau)}{2}.
$$
Since, 
$$
|u_{x_i}(x_\tau)|=|u_{x_i}(y_\tau)|= \tau\left|u_{x_i}\left(\frac{x_\tau+y_\tau}{2}\right)\right|\le \tau <1
$$
for $i=1,\dots,n$, we have
$$
\delta u(z) - \frac{1}{2}\tr\sigma\sigma^t\left(A + AD^2u(z) A + (z+ADu(z) )\otimes(z+ADu(z) ) \right)=0,\quad z=x_\tau, y_\tau.
$$
Set $z_\tau=(x_\tau+y_\tau)/2$, $p_\tau=Du(z_\tau)$, and notice
\begin{eqnarray}
\delta {\mathcal C}^\tau(x,y)
&\le& 
\frac{(\tau -1)}{2}\tr\sigma\sigma^tA + \frac{1}{2}\tr\sigma\sigma^t\left[A\left(\tau D^2u(z_\tau)-\frac{D^2u(x_\tau)+ D^2u(y_\tau)}{2} \right)A\right] \nonumber \\
&& + \frac{\tau}{2}|\sigma^t(z_\tau + ADu(z_\tau))|^2 -\frac{1}{4}|\sigma^t(x_\tau + ADu(x_\tau))|^2-\frac{1}{4}|\sigma^t(y_\tau + ADu(y_\tau))|^2 \nonumber \\
&\le & \frac{(\tau -1)}{2}\tr\sigma\sigma^tA + \frac{(\tau -1)}{4}(|\sigma^tx_\tau|^2+|\sigma^ty_\tau|^2) +\frac{1}{2}\tau(1-\tau)|\sigma^tAp_\tau|^2\nonumber \\
&\le & \frac{(\tau -1)}{2}\tr\sigma\sigma^tA + \frac{n}{2}\tau(1-\tau)|\sigma^tA|^2\nonumber,
\end{eqnarray}
for each $x,y\in\R^n.$ Sending $\tau\rightarrow 1^-$, we conclude that 
$$
 u\left(\frac{x+y}{2}\right)-\frac{u(x) + u(y)}{2}\le 0, \quad x,y\in\R^n.
$$
\par 3. To make this argument rigorous, we fix $0< \tau  <1$ and now define 
$$
w^\tau(x,y,z)=\tau u\left(z\right)-\frac{u(x) + u(y)}{2}, \quad x,y,z\in\R^n;
$$
and for $\eta>0$, set

$$
\varphi_\eta(x,y,z)=\frac{1}{2\eta}\left|z-\frac{x+y}{2} \right|^2, \quad x,y,z\in\R^n.
$$
Notice that
\begin{eqnarray}
(w^\tau -\varphi_\eta)(x,y,z)&=& \tau \left\{u\left(z\right) - u\left(\frac{x+y}{2}\right)\right\} - \frac{1}{2\eta}\left|z-\frac{x+y}{2} \right|^2\nonumber 
                                                            +{\cal C}^\tau(x,y) \\
                                                             &\le & \sqrt{n}\tau \left|z-\frac{x+y}{2} \right| -  \frac{1}{2\eta}\left|z-\frac{x+y}{2} \right|^2+ {\cal C}^\tau(x,y).
\end{eqnarray}
From our arguments in part 1 above, it follows that 
$$
\lim_{|x|+|y|+|z|\rightarrow \infty}(w^\tau-\varphi_\eta)(x,y,z)=-\infty
$$
and, in particular, that there is $(x_\eta,y_\eta,z_\eta)\in \R^n\times \R^n\times \R^n$ maximizing $w^\tau-\varphi_\eta.$  Now it is possible to 
argue as we did in the proof of Proposition \ref{Newcomparison} to conclude that ${\cal C}^\tau(x,y)\le O(1-\tau)$, as $\tau\rightarrow 1^-.$ See also the 
proof of Lemma 3.7 in \cite{Hynd2}.
\end{proof}

\begin{cor}\label{OmegaDeltaBD}
There is a constant $C=C(A)>0$, independent of $0<\delta <1$, such that if $|x|\ge C$ and $p\in J^{1,-}u_\delta(x)$, then $\max_{1\le i\le n}|p_i|\ge 1$. In particular, if 
$Du_\delta(x)$ exists and $\max_{1\le i\le n}|\partial_{x_i}u_\delta(x)|<1$, then $|x|<C$. 
\end{cor}
\begin{proof}
Choose $C=C(A)$ so large that 
$$
\delta u_\delta(z)< \frac{1}{2}\tr \sigma\sigma^tA+\frac{1}{2}(|\sigma^tz|-\sqrt{n}|A|)^2, \quad |z|\ge C
$$
for $0<\delta<1$.  Recall that $J^{1,-}u_\delta(x)=\underline{\partial}u_\delta(x)$ by the convexity of $u_\delta$ (see Proposition 4.7 in \cite{BC}); here $\underline{\partial}u_\delta(x):=\{p\in \R^n: u_\delta(y)\ge u_\delta(x)+ p\cdot (y-x), \; \text{for all}\; y\in \R^n\}$. Moreover, $(p,0)\in J^{2,-}u_\delta(x)$,
and so 
$$
\max_{1\le i\le n}\{\delta u_\delta(x) +G(0,p,x;A), |p_i|-1\}\ge 0.
$$
As 
$$
\delta u_\delta(x) +G(0,p,x;A)\le \delta u_\delta(x) -\frac{1}{2}\tr \sigma\sigma^tA-\frac{1}{2}(|\sigma^tx|-\sqrt{n}|A|)^2<0,$$
it must be that $\max_{1\le i\le n}|p_i|\ge 1$.
\end{proof}

The primary importance of following corollary is in establishing a useful lower bound on $u_\delta$. The lower bound it establishes is key 
in proving the existence of an eigenvalue. 
\begin{cor}\label{NewLipLem}
There is a constant $C=C(A)>0$, independent of $0<\delta<1$, such that 
\begin{equation}\label{NewLipForm}
u_\delta(x)= \min_{|y|\le C}\left\{u_\delta(y) + \sum^n_{i=1}|x_i-y_i|\right\}, \; x\in \R^n
\end{equation}
and 
\begin{equation}\label{LowerBoundLemma}
u_\delta(x)\ge u_\delta(0) + \left(\sum^n_{i=1}|x_i|-C\right)^+, \; x\in \R^n. 
\end{equation}
\end{cor}

\begin{proof}
Set $v$ to be the right hand side of \eqref{NewLipForm}. As $\max_{1\le i\le n}|\partial_{x_i}u_\delta|\le 1$ 
$$
u_\delta \le v,
$$
for any $C>0$. In particular, choosing $y=x$ in \eqref{NewLipForm} gives $u_\delta=v$ for $|x|\le C$.  Now select $C=C(A)$ such that 
$$
K+ \sum^n_{i=1}|x_i| - \frac{1}{2}\tr\sigma\sigma^tA- \frac{1}{2}(|\sigma^tx|-\sqrt{n}|A|)^2\le 0\quad \text{for}\quad |x|\ge C,
$$
where $K$ is the constant appearing in the definition of $\bar{u}$ in equation \eqref{Newviscsuper}.  

\par It is clear that $\max_{1\le i\le n}|v_{x_i}|\le 1$ and straightforward to verify that  as $u_\delta$ is convex, $v$ is convex, as well. 
Now let $(p,X)\in  J^{2,+}v(x_0)$.  If $|x_0|<C$, then $v=u_\delta$ in a neighborhood of $x_0$ and so 
$$
\max_{1\le i\le n}\left\{\delta v(x_0) +G(D^2v(x_0),Dv(x_0),x_0;A) , |p_i|- 1\right\}\le 0. 
$$
If $|x_0|\ge C$, then by the convexity of $v$
\begin{eqnarray}
\delta v(x_0) +G(D^2v(x_0),Dv(x_0),x_0;A)
&\le &\delta\left(u(0) + \sum^n_{i=1}|x_0\cdot e_i|\right)-\frac{1}{2}\tr\sigma\sigma^tA  -\frac{1}{2}(|\sigma^tx_0|-\sqrt{n}|A|)^2 \nonumber \\
&\le &K + \sum^n_{i=1}|x_0\cdot e_i|-\frac{1}{2}\tr\sigma\sigma^tA-\frac{1}{2}(|\sigma^tx_0|-\sqrt{n}|A|)^2\le 0 \nonumber 
\end{eqnarray}
while we always have $\max_{1\le i\le n}|p_i|\le 1$. Therefore, $v$ is a subsolution of \eqref{AdeltaPDE} with growth \eqref{Newugrowth}, and consequently 
$$
v\le u_\delta.
$$
This verifies \eqref{NewLipForm}. \par By Theorem \ref{yeahUnique}, $u_\delta(x)=u_\delta(-x)$ for all $x\in \R^n$.  As $u_\delta$ is convex, $u_\delta(x)=(u_\delta(x)+u_\delta(-x))/2\ge u_\delta(0)$ for all $x\in \R^n$. Thus, \eqref{LowerBoundLemma} follows from \eqref{NewLipForm}. 
\end{proof}

\subsection{Existence}
We assume that $A$ is a fixed symmetric, $n\times n$  matrix and will now establish the existence of a unique eigenvalue $\lambda(A)$.  To this end, we define
$$
\begin{cases}
\lambda_\delta:=\delta u_\delta(0)\\
v_\delta(x):=u_\delta(x)-u_\delta(0)
\end{cases}.
$$
Notice that for $\underline{u}$ and $\overline{u}$ defined in \eqref{Newviscsub} and \eqref{Newviscsuper}, $\underline{u}(0)\le u_\delta(0)\le \overline{u}(0)$. Hence, 
$$
\frac{1}{2}\tr\sigma\sigma^tA\le \lambda_\delta \le K.
$$ 
It is also clear that $v_\delta $ satisfies
$$
\begin{cases}
|v_\delta(x)|\le  \sum^n_{i=1}|x_i|\\
|v_\delta(x)-v_\delta(y)|\le  \sum^n_{i=1}|x_i-y_i|\\
\end{cases} \quad x,y\in\R^n.
$$


\begin{lem}\label{ExisLem}
There is a sequence $\delta_k>0$ tending to $0$ as $k\rightarrow \infty$, $\lambda(A)\in \R$, and $u\in C(\R^n)$ with $|u(x)-u(y)|\le  \sum^n_{i=1}|x_i-y_i|$ such 
that 
\begin{equation}\label{NewvEst}
\begin{cases}
\lambda(A)=\lim_{k\rightarrow \infty}\lambda_{\delta_k}\\
v_{\delta_k}\rightarrow u \;\text{in}\;\text{locally uniformly as}\; k\rightarrow \infty 
\end{cases}.\nonumber
\end{equation}
Moreover,  $u$ is a convex solution of \eqref{CorrectorPDE} with eigenvalue $\lambda(A)$ that satisfies \eqref{Newugrowth}.
\end{lem}

\begin{proof}
It is immediate that $\lambda(A)=\lim_{k\rightarrow \infty}\lambda_{\delta_k}$ for some $\delta_k\rightarrow 0$, as $\lambda_\delta$ is bounded. The convergence assertion of a subsequence $v_{\delta_k}$ to some $u$, locally uniformly in $\R^n$, follows from the Arzel\`{a}-Ascoli theorem and a routine diagonalization argument; it is
clear that $|u(x)-u(y)|\le  \sum^n_{i=1}|x_i-y_i|$ and that $u$ is convex.  It also follows easily from the convergence assertion and the stability properties of viscosity solutions (Lemma 6.1 of \cite{CIL}) that 
$u$ satisfies the PDE \eqref{CorrectorPDE} in the sense of viscosity solutions.  As $|u(x)|\le \sum^n_{i=1}|x_i|$ for all $x\in \R^n$, 
$$
\limsup_{|x|\rightarrow \infty}\frac{u(x)}{\sum^n_{i=1}|x_i|}\le 1.
$$
By \eqref{LowerBoundLemma}, for all $|x|$ sufficiently large
$$
v_{\delta}(x)=u_\delta(x)-u_\delta(0)\ge \sum^n_{i=1}|x_i|-C,
$$
for some $C$ {\it independent} of $0<\delta<1$. Thus,

$$
\liminf_{|x|\rightarrow \infty}\frac{u(x)}{\sum^n_{i=1}|x_i|}\ge 1,
$$
and so $u$ satisfies \eqref{Newugrowth}, as well.
\end{proof}
Moreover, when $\det A\neq 0$, $u$ is a viscosity solution of the PDE of the form
\begin{equation}\label{modelPDE}
\max\left\{\lambda(A)+G(D^2u,Du,x;A),H(Du)\right\}=0,\; x\in\R^n \nonumber
\end{equation}
where $\lambda(A)+G(D^2\psi,D\psi,x;A)$ is a semi-linear, {\it uniformly} elliptic operator. Moreover, $G(X,p,x;A)$ depends quadratically on $p$
and $H$ is convex. It follows from a minor modification of proof of part Proposition 3.1 in \cite{Hynd1} that there is a constant  $C=C(O,\alpha)$ such that that 
$|u|_{C^{1,\alpha}(O)}\le C$ for each open bounded $O\subset \R^n$ and $\alpha\in (0,1)$. This completes the proof of Theorem \ref{firstmainthm}.

\begin{cor}\label{OmegaCor} Assume $\det A\neq 0$ and let $u$ be as described in the statement of Lemma \eqref{ExisLem}. Then
\begin{equation}\label{DEFOm}
\Omega:=\left\{x\in\R^n: \max_{1\le i\le n}|u_{x_i}(x)|<1 \right\}\nonumber
\end{equation}
is open and bounded. Moreover, $u\in C^\I(\Omega)$
\end{cor}

\begin{proof}  The first assertion follows immediately from Corollary \ref{OmegaDeltaBD} and since $x\mapsto Du(x)$ is continuous mapping of $\R^n$ into itself. 
The second assertion follows from standard elliptic regularity, as $u$ satisfies a semilinear, uniformly elliptic PDE on $\Omega$  (see Theorem 6.17 \cite{GT}).  
\end{proof}

\section{Properties of the eigenvalue function}\label{EigProp}
In view of Theorem \ref{firstmainthm}, the solution of the eigenvalue problem defines a function $\lambda: {\mathcal S}(n)\rightarrow \R$.  Below, we prove
 some important properties of $\lambda$.  
Our basic tool will be the comparison principle described in Proposition \ref{Newcomparison}.  We use this property
to show that $\lambda$ is a monotone, convex function. Moreover, the regularity assertion of Theorem \ref{firstmainthm} will be employed to establish  minmax formulae for $\lambda$.  Our main 
result is as follows. 

\begin{thm}\label{lampropthm}  Let $\lambda : {\mathcal S}(n)\rightarrow \R$ be as described in the statement  of Theorem \ref{firstmainthm}. Then \\
(i) $\lambda$ is nondecreasing,\\
(ii) $\lambda$ is convex, \\ 
(iii) for each $A\in S(n)$ and each permutation matrix $U$, 
$$
\lambda(UAU^t)=\lambda(A),
$$
and (iv) $\lambda_- \le \lambda\le \lambda_+$, where
\begin{eqnarray}
\lambda_{-}(A)&=&\sup\left\{\inf_{x\in\R^n}\left(-G(D^2\phi(x),D\phi(x),x;A)\right):  \phi\in C^2(\R^n), \max_{1\le i\le n}|\phi_{x_i}|\le 1\right\}\nonumber
\end{eqnarray} 
and 
\begin{eqnarray}
\lambda_{+}(A)&=&\inf\left\{ \sup_{\max_i|\psi_{x_i}(x)|<1}\left(-G(D^2\psi(x),D\psi(x),x;A)\right): \psi\in C^2(\R^n),\liminf_{|x|\rightarrow \infty}\frac{\psi(x)}{\sum^n_{i=1}|x_i|}\ge 1\right\}.\nonumber
\end{eqnarray} 
Furthermore, $\lambda_-(A)=\lambda(A)= \lambda_+(A)$, provided $\det A\neq 0.$
\end{thm}

In order to establish these properties, we will make use of the following characterizations of $\lambda$, which follow immediately from the existence and uniqueness of the eigenvalue function. 
The following formulae, manifestations of the comparison principle, will be used below to establish monotone upper and lower bounds on the eigenvalue that will be crucial
to deduce the other properties listed in the above theorem. 
\begin{lem} Let $A\in {\mathcal S}(n)$ and assume that $\lambda(A)$ is the solution of the eigenvalue problem associated with equation \eqref{CorrectorPDE}.  Then
\begin{eqnarray}\label{Newmaxform}
\lambda(A)&=&\sup\text{{\huge \{}} \lambda \in\R: \text{there exists a subsolution $u$ of \eqref{CorrectorPDE} with eigenvalue $\lambda$}, \nonumber\\
&& \left. \hspace{1.6in} \text{satisfying $\limsup_{|x|\rightarrow \infty}\frac{u(x)}{\sum^n_{i=1}|x_i|}\le 1$.}  \right\}
\end{eqnarray}
and
\begin{eqnarray}\label{Newminform}
\lambda(A)&=&\inf\text{{\huge \{}} \mu \in\R: \text{there exists a supersolution $v$ of \eqref{CorrectorPDE} with eigenvalue $\mu$}, \nonumber\\
&& \left. \hspace{1.7in} \text{satisfying $\liminf_{|x|\rightarrow \infty}\frac{v(x)}{\sum^n_{i=1}|x_i|}\ge 1$.}  \right\}.
\end{eqnarray}
\end{lem}

\subsection{Monotone upper and lower bounds}
In this subsection, we prove that $\lambda$ is a locally bounded, nondecreasing, convex function and therefore it is necessarily continuous.  We first show that the function
$\lambda$ is bounded above and below by monotone functions that are constructed from $\lambda_1:\R\rightarrow \R$, the solution of the eigenvalue problem found by Barles and Soner 
\cite{BS}.  Then we show $\lambda$ is convex by an elementary argument. It turns out that any convex function that is bounded above by a nondecreasing function is necessarily nondecreasing itself, and therefore we will be able to conclude that $\lambda$ is monotone.  This implies, in particular, that the PDE $\psi_t+\lambda(d(p)D^2\psi d(p))=0$ is backwards parabolic which will be useful to us in the following section.

\begin{prop}\label{LamMonBounds}
There are monotone, non-decreasing functions $\underline{\lambda}, \bar{\lambda}: {\mathcal S}(n)\rightarrow \R$
such that 
$$
\underline{\lambda}(A)\le \lambda(A) \le  \bar{\lambda}(A), \quad A\in {\mathcal S}(n).
$$
\end{prop}
 \begin{proof} 1. By appendix A of \cite{BS}, for each $a\in \R$ and $\alpha>0$, there is a unique $\lambda=\lambda_1(a,\alpha)\in\R$ such that the ODE
\begin{equation}\label{NewlamODEalpha}
\max\left\{\lambda - \frac{\sigma^2}{2}(a+a^2u'' + (x+au')^2), |u'|-\alpha\right\}=0, \quad x\in \R \nonumber
\end{equation}
has a solution $u=u_1(\cdot\; ; a, \alpha)\in C(\R)$ satisfying
\begin{equation}\label{ugrowthalph}
\lim_{|x|\rightarrow \infty}\frac{u(x)}{|x|}=\alpha.\nonumber
\end{equation}
When $a\neq 0$, $u_1(\cdot\; ; a, \alpha)\in C^2(\R)$. Moreover, the function $a\mapsto \lambda_1(a,\alpha)$ is continuous and monotone non-decreasing for each $\alpha >0.$  
\par 2. Now write $\sigma^t A\sigma=P\Lambda P^t$, where $PP^t=I_n$ and $\Lambda=\text{diag}(a_1, \dots, a_n)$. Next, define 
$$
\underline{\lambda}(A):=\sum^{n}_{i=1}\lambda_1(a_i, 1/|\sigma|\sqrt{n})
$$
and 
$$
\underline{u}(x;A):=\sum^{n}_{i=1}u_1(x\cdot \sigma Pe_i; a_i,1/|\sigma|\sqrt{n}).
$$
When $\det A\neq 0$, a direct computation shows $\underline{u}$ is a subsolution of equation \eqref{CorrectorPDE} with eigenvalue $\underline{\lambda}$; the general case 
then follows by straightforward limiting arguments and the stability of viscosity solutions under local uniform convergence. 
\par 3.  Likewise, we define 
$$
\overline{\lambda}(A):=\sum^{n}_{i=1}\lambda_1(a_i, |\sigma^{-1}|\sqrt{n})
$$
and 
$$
\overline{u}(x;A):=\sum^{n}_{i=1}u_1(x\cdot \sigma Pe_i; a_i,|\sigma^{-1}|\sqrt{n}).
$$
As above, $\overline{u}$ is a supersolution of equation \eqref{CorrectorPDE} with eigenvalue $\overline{\lambda}$ satisfying 
$$
\liminf_{|x|\rightarrow \infty}\frac{\overline{u}(x)}{\sum^n_{i=1}|x_i|}\ge 1. 
$$
The desired conclusion then follows from formulae \eqref{Newmaxform} and \eqref{Newminform}. 
\end{proof} 

\par Now we turn to the regularity properties of $\lambda$ and show $\lambda$ is convex and necessarily continuous. As mentioned, this fact will be 
used to show that $\lambda$ is monotone nondecreasing.

\begin{prop}
$\lambda: {\cal S}(n)\rightarrow \R$ is convex. 
\end{prop}

\begin{proof} 1. Let $A_1, A_2\in {\cal S}(n)$ and set $A_3:=(A_1+A_2)/2$. We will first show 
$$
\lambda_3\le \frac{\lambda_1+\lambda_2}{2}.
$$
where $\lambda_i:=\lambda(A_i)$, $i=1,2,3.$  Let $u_i=u(\cdot; A_i)$ and assume $u_i\in C^2(\R^n)$; the general argument follows 
from standard viscosity solutions methods.  Finally, we also assume $\frac{1}{2}\sigma\sigma^t=I_n$.  A simple inspection of the reasoning below will convince the reader that 
this can be done without any loss of generality. 

\par Fix $\tau>0$. Note that the function 
$$
\R^n\times\R^n\ni (x_1,x_2)\mapsto \tau u_3(x_3)-\frac{u_1(x_1)+u_2(x_2)}{2},  \quad x_3:=(x_1+x_2)/2
$$
has a maximum on $\R^n\times\R^n$, by adapting the proof of Proposition \ref{PropConvexUdelta}. For simplicity, we denote this point by $(x_1,x_2)$ and suppress the $\tau$ dependence. As $(x_1,x_2)$ is a maximizer,
\begin{equation}\label{firstOrderConv}
\tau Du_3(x_3)=Du_1(x_1)=Du_2(x_2)
\end{equation}
and
$$
\left[
\begin{array}{cc}
\frac{\tau}{4}D^2u_3(x_3)-\frac{1}{2}D^2u_1(x_1) &\frac{\tau}{4}D^2u_3(x_3)\\
\frac{\tau}{4}D^2u_3(x_3)& \frac{\tau}{4}D^2u_3(x_3)-\frac{1}{2}D^2u_2(x_2)
\end{array}
\right]\le 0.
$$
The above matrix inequality implies that
$$
\frac{\tau}{4}D^2u_3(x_3)(\xi_1+\xi_2)\cdot (\xi_1+\xi_2)\le \frac{1}{2}D^2u_1(x_1)\xi_1\cdot \xi_1 +\frac{1}{2}D^2u_2(x_2)\xi_2\cdot \xi_2
$$
for each $\xi_1,\xi_2\in\R^n.$  As each $u_i$ is convex, an application of the Cauchy-Schwarz inequality for nonnegative-definite, symmetric matrices gives
\begin{equation}\label{ImportantConv}
\tau \tr\left[A_3 D^2u_3(x_3)A_3\right]\le \frac{1}{2}\tr\left[A_1D^2u_1(x_1)A_1\right]+\frac{1}{2}\tr\left[A_2D^2u_2(x_2)A_2\right].
\end{equation}

\par 2.  By \eqref{firstOrderConv},
$$
|\partial_{x_i}u_1(x_1)|=|\partial_{x_i}u_2(x_2)|=\tau |\partial_{x_i} u_3(x_3)|\le \tau<1,\quad i=1,\dots,n,
$$
and so
$$
\lambda_i +G(D^2u_i,Du_i,x;A_i)=0, \quad i=1,2.
$$
Therefore, using inequality \eqref{ImportantConv} and a bit of algebra provides 
\begin{align*}
\tau \lambda_3 -\frac{\lambda_1+\lambda_2}{2}
 &\le  (\tau-1)\tr A_3 +\tau|x_3|^2 -\frac{1}{2}|x_1|^2-\frac{1}{2}|x_2|^2 + \tau|A_3Du_3(x_3)|^2- \frac{1}{2}|A_1Du_1(x_1)|^2 \\
 &\quad - \frac{1}{2}|A_2Du_2(x_2)|^2 +2\tau x_3\cdot A_3Du_3(x_3)-x_1\cdot A_1Du_1(x_1)-x_2\cdot A_2Du_2(x_2). \\
\end{align*}
We also have
$$
\tau|x_3|^2 -\frac{1}{2}|x_1|^2-\frac{1}{2}|x_2|^2=(\tau-1)|x_3|^2 -\left|\frac{x_1-x_2}{2}\right|^2,
$$
and using the first order conditions \eqref{firstOrderConv} 
\begin{align*}
 \tau|A_3Du_3(x_3)|^2- \frac{1}{2}|A_1Du_1(x_1)|^2- \frac{1}{2}|A_2Du_2(x_2)|^2 & = \tau(1-\tau)\frac{|A_1 Du_3(x_3)|^2+|A_2 Du_3(x_3)|^2}{2}\\
 & \quad -\left| \tau\left(\frac{A_1-A_2}{2}\right)Du_3(x_3)\right|^2
\end{align*}
and 
\begin{align*}
2\tau x_3\cdot A_3Du_3(x_3)-x_1\cdot A_1Du_1(x_1)-x_2\cdot A_2Du_2(x_2)&=\left(\frac{A_1-A_2}{2}\right)\left(\frac{x_2-x_1}{2}\right)\cdot \tau Du_3(x_3)\\
&\le \left| \tau\left(\frac{A_1-A_2}{2}\right)Du_3(x_3)\right|^2+\left|\frac{x_1-x_2}{2}\right|^2.
\end{align*}
Combining the previous four inequalities lead us to
$$
\tau\lambda_3 -\frac{\lambda_1+\lambda_2}{2}\le (\tau-1)\tr A_3 +  (\tau-1)|x_3|^2+\tau(1-\tau)\frac{|A_1 Du_3(x_3)|^2+|A_2 Du_3(x_3)|^2}{2}\le C(1-\tau)
$$
for some universal constant $C$. We conclude by letting $\tau\rightarrow 1^-.$

\par 3. Finally, we remark that virtually the same steps can be used to show 
$$
\lambda(s A+(1-s)B)\le s \lambda(A)+(1-s)\lambda(B),
$$
for $A,B\in {\cal S}(n)$ and $0\le s \le 1.$ Therefore, the argument above which shows that $\lambda$ is midpoint convex also can be used to show $\lambda$ is 
convex. 
\end{proof}

\begin{cor}
$\lambda: {\cal S}(n)\rightarrow \R$ is continuous. 
\end{cor}

\subsection{Symmetry, monotonicity, and min-max formulae}
We have shown that $\lambda$ is convex. In order to complete the proof of Theorem \ref{lampropthm},  we need prove assertions $(i)$, $(iii)$ and $(iv)$. To this end, 
we shall make use of the convexity of $\lambda$, the formulae \eqref{Newmaxform} and \eqref{Newminform} and the regularity of solutions of equation
\eqref{CorrectorPDE}.  

\begin{proof}( Theorem \ref{lampropthm} $(i)$)
Let $A\in {\cal S}(n)$ and $u(\cdot; A)$ be a solution of \eqref{CorrectorPDE} with eigenvalue $\lambda(A).$ Direct computation has that $v(x):=u(U^t x; A)$ is a viscosity solution of the PDE
$$
\max_{1\le i\le n}\left\{\lambda(A)+G(D^2v,Dv,x;UAU^t), |v_{x_i}|-1\right\}=0, \quad x\in \R^n
$$
that satisfies the usual growth condition \eqref{Newugrowth} for any permutation matrix $U$. The key observation here is that $v_{x_i}=Du(U^tx)\cdot U^t e_i$ and $U^t$ permutes the standard basis vectors.  By Proposition \ref{Newcomparison}, we have $\lambda(A)=\lambda(UAU^t).$
\end{proof}

\begin{proof} (Theorem \ref{lampropthm} $(iii)$)
It suffices to verify the general assertion: if $f,g: {\cal S}(n)\rightarrow \R$ with $g$ nondecreasing, $f$ convex, and $f\le g$, then $f$ is nondecreasing. In our case, $f=\lambda$ and $g=\overline{\lambda}$ from Proposition \ref{LamMonBounds}.

\par Suppose that $Q\in \partial f(A_0)\neq \emptyset$; that is, $Q\in {\cal S}(n)$ and
\begin{equation}\label{subdiffA}
Q\cdot (A-A_0)+f(A_0)\le f(A), \quad A\in {\cal S}(n).
\end{equation}
We claim $Q\ge 0$. To see this, let $\xi \in \R^n$ and set
$$
A:=A_0-t\xi\otimes\xi 
$$
for  $t>0$.  As $g$ is nondecreasing, substituting this $A$ in \eqref{subdiffA} gives
$$
-tQ\xi\cdot \xi + f(A_0)\le f(A_0-t\xi\otimes\xi )\le g(A_0-t\xi\otimes\xi)\le g(A_0).
$$
Clearly this inequality holds for all $t>0$ if and only if $Q\xi\cdot \xi\ge 0$. As a result, $Q\ge 0$ and thus $f$ is nondecreasing.  
\end{proof}

\begin{proof}(of Theorem \ref{lampropthm} $(iv)$) 1. Fix $A\in {\mathcal S}(n)$, let $\phi\in C^2$ and suppose that $\max_i|\phi_{x_i}|\le 1$. Now set 
$$
\mu^\phi(A):=\inf_{x\in \R^n}(-G(D^2\phi(x), D\phi(x),x;A)).
$$
If $\mu^\phi(A)=-\I$, then $\mu^\phi(A)\le \lambda(A);$ if $\mu^\phi(A)>-\I$, by the assumptions on $\phi$ and the definition of $\mu^\phi(A)$ 
$$
\max_{1\le i\le n}\left\{\mu^\phi(A)+G(D^2\phi(x), D\phi(x),x;A), |\phi_{x_i}|-1 \right\}\le 0.
$$
By \eqref{Newmaxform}, we still have $\mu^\phi(A)\le \lambda(A).$ Thus, $\lambda_-(A)=\sup_\phi\mu^\phi(A)\le \lambda(A).$

\par  2. Again fix $A\in {\mathcal S}(n)$. Now let $\psi\in C^2$ satisfy $\liminf_{|x|\rightarrow }\psi(x)/\sum^n_{i=1}|x_i|\ge 1$ and set 
$$
\tau^\psi(A):=\sup_{\max_i|\psi_{x_i}(x)|<1}(-G(D^2\psi(x), D\psi(x),x;A)).
$$
If $\tau^\psi(A)=+\I$, then $\tau^\psi(A)\ge \lambda(A);$ if $\tau^\psi(A)<+\I$, by the assumptions on $\psi$ and the definition of $\tau^\psi(A)$ 
$$
\max\left\{\tau^\psi(A)+G(D^2\psi(x), D\psi(x),x;A), |\psi_{x_i}|-1 \right\}\ge 0.
$$
By \eqref{Newminform}, we still have $\tau^\psi(A)\ge \lambda(A).$ Hence, $\lambda_+(A)=\inf_{\psi}\tau^\psi(A)\ge \lambda(A).$
 
\par 3. Suppose that $\det A\neq 0$ and let $u=u(\cdot, A)$ be a convex solution of \eqref{CorrectorPDE} associated to $\lambda(A)$ that satisfies $Du\in C^\alpha_{\text{loc}}(\R^n)\cap L^\infty(\R^n)$ for each $\alpha\in(0,1)$. We first claim 
that $\lambda(A)\le \lambda_-(A)$.  To see this we mollify $u$, $u^\epsilon:=\eta^\epsilon*u$
(see Appendix C of \cite{Evans} for more on mollification).  A straightforward computation implies
$$
\eta^\epsilon*(x\mapsto |\sigma^t(x+ADu(x))|^2)=|\sigma^t(x+ADu^\epsilon)|^2 + O(\epsilon^\alpha)
$$
as $\epsilon\rightarrow 0,$ for $x$ belonging to bounded subdomains of $\R^n.$ Therefore, as $u$ solves the PDE \eqref{CorrectorPDE} almost everywhere on $\R^n$ 
\begin{equation}\label{lamdalpha}
\lambda(A)+G(D^2u^\epsilon,Du^\epsilon,x;A)\le O(\epsilon^\alpha)
\end{equation}
for $x$ belonging to bounded subdomains of $\R^n.$  Furthermore, the convexity of $u^\epsilon$ and uniform boundedness of $|Du^\epsilon|_{L^\infty(\R^n)}$ imply that 
\eqref{lamdalpha} actually holds for all $x\in \R^n$.  Consequently, 
\begin{equation}
\lambda_-(A)\ge  \inf_{\R^n}(-G(D^2u^\epsilon,Du^\epsilon,x;A))\ge\lambda(A)+O(\epsilon^\alpha).
\end{equation}

\par 4. Next, we claim that $\lambda(A)\ge \lambda_+(A)$.   An important observation for us will be that $\max_{1\le i\le n}|u_{x_i}|$ is uniformly continuous on $\R^n$. This is due to
$$
\lim_{|x|\rightarrow \infty}\max_{1\le i\le n}|u_{x_i}(x)|=1,
$$
which in turn follows from the limit \eqref{Newugrowth} and the fact that $u$ is convex. It now follows that $\max_{1\le i\le n}|u^\epsilon_{x_i}|$ converges to 
$\max_{1\le i\le n}|u_{x_i}|$ uniformly on $\R^n$.

\par Set
$$
u^{\epsilon,\delta}:=(1+\delta)u^\epsilon,
$$
where $\delta>0$ is fixed, and notice that 
$$
\max_{1\le i\le n}|u_{x_i}^{\epsilon,\delta}(x)|<1 \Leftrightarrow \max_{1\le i\le n}|u_{x_i}^\epsilon(x)|<\frac{1}{1+\delta}.
$$
As $1/(1+\delta)<1$, there is $\rho=\rho(\delta)>0$ so small such that 
$$
\gamma:=\frac{1}{1+\delta}+\rho<1.
$$
Also, for $\epsilon_0=\epsilon_0(\delta)>0$ small enough
$$
\max_{1\le i\le n}|u_{x_i}(x)|\le \max_{1\le i\le n}|u_{x_i}^\epsilon(x)| + \rho< \frac{1}{1+\delta}+\rho=\gamma,
$$
provided $0<\epsilon<\epsilon_0$  and $x$ satisfies $\max_{1\le i\le n}|u_{x_i}^{\epsilon,\delta}(x)|<1$.  Moreover, there is $\epsilon_1=\epsilon_1(\delta)$ such that 
$$
\{x\in \R^n: \max_{1\le i\le n}|u_{x_i}(x)|<\gamma\}\subset \Omega_\epsilon:=\{x\in \Omega: \text{dist}(x,\partial\Omega)>\epsilon\}
$$
for $0<\epsilon<\epsilon_1$.  This inclusion follows as the set $\{x\in \R^n: \max_{1\le i\le n}|u_{x_i}(x)|<\gamma\}$ is an open subset of the open set $\Omega$ (Corollary \ref{OmegaCor}).  

\par Hence for $0<\epsilon<\min\{\epsilon_0,\epsilon_1\}$, we have 
$$
\left\{x\in\R^n: \max_{1\le i\le n}|u_{x_i}^{\epsilon,\delta}(x)|<1\right\}\subset \Omega_\epsilon.
$$
In particular, if $\max_{1\le i\le n}|u_{x_i}^{\epsilon,\delta}(x)|<1$, then 
$$
\lambda(A)+G(D^2u^{\epsilon},Du^{\epsilon},x;A)=O(\epsilon^\alpha)
$$
as $\lambda(A)+G(D^2u,Du,x;A)=0$,  a.e. on $\Omega$.

\par With the above computations, and the fact that $u\in C^\infty(\Omega)$ gives 
\begin{align*}
\lambda_+(A)& 
\le \sup_{\Omega_\epsilon}(-G(D^2u^{\epsilon},Du^{\epsilon},x;A)) + O(\delta)\le \lambda(A) +O(\epsilon^\alpha) + O(\delta)
\end{align*}
for $\epsilon\in (0,\min\{\epsilon_0,\epsilon_1\})$. We conclude by first sending $\epsilon\rightarrow 0^+$ and then $\delta\rightarrow 0^+$.
\end{proof}

\section{Convergence}\label{zConverge}
In this section, we verify Theorem \ref{lastmainthm} which characterizes $\lim_{\epsilon\rightarrow 0^+}z^\epsilon$ as a solution of the 
nonlinear diffusion equation 
\begin{equation}\label{SecPsiEq}
\psi_t+\lambda(d(p)D^2\psi d(p))=0, \quad (t,p)\in (0,T)\times (0,\infty)^n.
\end{equation} 
Here, $\lambda: {\cal S}(n)\rightarrow \R$ is of course the solution of the eigenvalue problem discussed in previous sections. The method of proof is relatively standard in the theory of viscosity solutions and goes as follows. We show the upper limit
$$
\overline{z}(t,p,y):=\limsup_{\substack{\epsilon\rightarrow 0^+\\ (t',p',y')\rightarrow (t,p,y)}}z^\epsilon(t',p',y')
$$
is a viscosity subsolution of \eqref{SecPsiEq} and the lower limit 
$$
\underline{z}(t,p,y):=\liminf_{\substack{\epsilon\rightarrow 0^+\\ (t',p',y')\rightarrow (t,p,y)}}z^\epsilon(t',p',y')
$$
is a viscosity supersolution of \eqref{SecPsiEq}.  

\par As $\overline{z}$ and $\underline{z}$ agree at time $t=T$ and satisfy natural growth estimates for large values of $p$ (see Lemma \ref{TechnicalLemma}), we will be able to conclude
$$
\overline{z}\le \underline{z}.
$$
Combined with the definitions above, we will also have $\overline{z}=\underline{z}=:\psi$ and that $z^\epsilon\rightarrow \psi$ locally uniformly as $\epsilon\rightarrow 0$  (Remark 6.2 in \cite{CIL}).
First, let us make a basic observation.

\begin{lem}
$\overline{z}$ and $\underline{z}$ are independent of $y.$
\end{lem}
\begin{proof} 1. As $|z^\epsilon_{y_i}|\le \sqrt{\epsilon}p_i$ for $i=1,\dots,n$ in the sense of viscosity solutions,
\begin{equation}\label{LipschitZeps}
|z^\epsilon(t,p,y_1)-z^\epsilon(t,p,y_2)|\le \sum^n_{i=1}\sqrt{\epsilon}p_i|(y_1-y_2)\cdot e_i|, \quad (t,p)\in (0,T)\times (0,\infty)^n, \; y_1,y_2\in \R^n.
\end{equation}\
Therefore, for $y_1,y_2\in \R^n$ 
\begin{align*}
\overline{z}(t,p,y_1)-\overline{z}(t,p,y_2)
&\le \limsup_{\substack{\epsilon\rightarrow 0^+\\ (t',p',y'_1,y'_2)\rightarrow (t,p,y_1,y_2)}}\left\{z^\epsilon(t',p',y'_1)-z^\epsilon(t',p',y'_2)\right\}\\
&\le \limsup_{\substack{\epsilon\rightarrow 0^+\\ (t',p',y'_1,y'_2)\rightarrow (t,p,y_1,y_2)}}\left\{\sum^n_{i=1}\sqrt{\epsilon}p_i|(y_1'-y_2')\cdot e_i|\right\}&=0.
\end{align*}
Hence, $\overline{z}$ is independent of $y$. 
\par 2. As 
$$
(-\underline{z})(t,p,y):=\limsup_{\substack{\epsilon\rightarrow 0^+\\ (t',p',y')\rightarrow (t,p,y)}}(-z^\epsilon)(t',p',y')
$$
and $-z^\epsilon$ also satisfies \eqref{LipschitZeps}, we may similarly conclude that $\underline{z}$ is independent of $y$. 
\end{proof}
We are finally in position to prove Theorem \ref{lastmainthm}.  The technique we will use, known as the {\it perturbed test function method}, is due to Evans \cite{LCE} and was first applied to this framework by Barles and 
Soner \cite{BS}.  One difference with the option pricing problem in several assets is that we must work with nonsmooth
``correctors" i.e. viscosity solutions $u$ of equation \eqref{CorrectorPDE}.   We will employ a smoothing argument to overcome this difficulty.

\begin{proof} (of Theorem \ref{lastmainthm}) 
1. We first show that $\underline{z}$ is supersolution of \eqref{SecPsiEq}. Assume that $\underline{z} - \phi$ has a local minimum at some point $(t_0,p_0)\in (0,T)\times (0,\infty)^n$ and $\phi\in C^\I$; for definiteness, we suppose 
that
$$
(\underline{z} - \phi)(t,p)\ge (\underline{z} - \phi)(t_0,p_0), \quad (p_0,t_0)\in \overline{B_\tau}(p_0,t_0)
$$
for some ball $\overline{B_\tau}(p_0,t_0)\subset (0,\infty)^n\times (0,T)$. We must show 
\begin{equation}\label{psisuper}
-\phi_t(t_0,p_0)- \lambda(d(p_0)D^2\phi(t_0,p_0)d(p_0))\ge 0.
\end{equation}
By subtracting $(t,p)\mapsto\eta(|t-t_0|^2+|p-p_0|^2)$ from $\phi$ and later sending $\eta\rightarrow 0^+ $, we may assume that $(t_0,p_0)$ is a {\it strict} local minimum point for $\underline{z}-\phi$ in $\overline{B_\tau}(t_0,p_0)$ and also that 
\begin{equation}\label{detD2phineq0}
\det D^2\phi(t_0,p_0)\neq 0. \nonumber
\end{equation}
We fix $\delta\in (0,1)$, and set 
$$
\begin{cases}
A^\delta(t,p):=(1-\delta)^2d(p)D^2\phi(t,p)d(p)\\ 
A_0:=A^\delta(t_0,p_0) \\
x^{\epsilon,\delta}(t,p,y):=(1-\delta)d(p)\frac{D\phi(t,p)-y}{\sqrt{\epsilon}}\\
\phi^{\epsilon,\delta,\rho}(t,p,y):=\phi(t,p) + \epsilon u^\rho\left(x^{\epsilon,\delta}(t,p,y)\right)
\end{cases}
$$
for $(t,p,y)\in (0,T)\times (0,\infty)^n\times \R^n.$  Here $u^\rho:=\eta^\rho*u$ is the standard mollification of $u=u(\cdot;A_0)$, where $u$ is a convex solution of \eqref{CorrectorPDE} with eigenvalue $\lambda(A_0)$ that satisfies \eqref{Newugrowth} and $u\in C^{1,\alpha}_\text{loc}(\R^n)$
for any $0<\alpha<1$.

\par 2. We claim there is a sequence of positive numbers $\epsilon_k\rightarrow 0$ and local minimizers $(t_k,p_k,y_k)\in \overline{B_\tau}(t_0,p_0)\times\R^n$ of $z^{\epsilon_k}-\phi^{\epsilon_k,\delta,\rho}$ such that $(t_k,p_k)\rightarrow (t_0, p_0),$ as $k\rightarrow \infty.$  We will use the idea presented in appendix of \cite{BP} to prove this. 

Let $y_0\in \R^n$ be given and select a sequence $\epsilon_k\rightarrow 0$ and $(t'_k,p'_k,y'_k) \rightarrow (t_0,p_0,y_0)$ as $k\rightarrow \infty$ such that 
$$
(z^{\epsilon_k}-\phi^{\epsilon_k,\delta,\rho})(t'_k,p'_k,y'_k)\rightarrow (\underline{z}-\phi)(t_0,p_0)
$$
(recall $\underline{z}$ is independent of the $y$ variable). By estimate \eqref{zgrowth3} below (see Lemma \ref{TechnicalLemma}),
\begin{equation}\label{whereIusedzgrowth}
\liminf_{|y|\rightarrow \infty}\frac{(z^{\epsilon}-\phi^{\epsilon,\delta,\rho})}{\sum^n_{i=1}\sqrt{\epsilon}p_i|y_i|}\ge \delta >0
\end{equation}
locally uniformly in $(t,p)\in (0,T)\times (0,\infty)^n$ {\it and} all $\epsilon$ sufficiently small. Thus, $z^{\epsilon_k}-\phi^{\epsilon_k,\delta,\rho}$ has a minimum at some 
$$
(t_k,p_k,y_k)\in \overline{B_\tau}(t_0,p_0)\times\R^n
$$
for all $k$ sufficiently large.  Moreover, it must be that $y_k$ is a bounded sequence for if not then \eqref{whereIusedzgrowth} implies 
$$
(z^{\epsilon_k}-\phi^{\epsilon_k,\delta,\rho})(t_k,p_k,y_k)\rightarrow +\infty
$$
while
$$
(z^{\epsilon_k}-\phi^{\epsilon_k,\delta,\rho})(t_k,p_k,y_k)\le (z^{\epsilon_k}-\phi^{\epsilon_k,\delta,\rho})(t'_k,p'_k,y'_k)
$$
and the right hand side above is bounded from above.

\par Without loss of generality, we assume that $(t_k,p_k,y_k)\rightarrow (t_1,p_1,y_1)$, as $k\rightarrow \infty$.
Notice that 
\begin{align*}
(\underline{z}-\phi)(t_1,p_1)&\le \liminf_{k\rightarrow \infty}(z^{\epsilon_k}-\phi^{\epsilon_k,\delta,\rho})(t_k,p_k,y_k)\le \liminf_{k\rightarrow \infty}(z^{\epsilon_k}-\phi^{\epsilon_k,\delta,\rho})(t'_k,p'_k,y'_k)=(\underline{z}-\phi)(t_0,p_0).
\end{align*}
As $(t_1,p_1)\in \overline{B_\tau}(t_0,p_0)$, it must be that $(t_1,p_1)=(t_0,p_0)$.

\par 3. We have at the point $(t_k, p_k, y_k)$
\begin{eqnarray}
|\phi_{y_i}^{\epsilon_k,\delta,\rho}|& =  &|(1-\delta)\sqrt{\epsilon_k}(p_k\cdot e_i)u^\rho_{x_i}(x^{\epsilon_k,\delta})|\le (1-\delta)\sqrt{\epsilon_k}(p_k\cdot e_i)< \sqrt{\epsilon_k}(p_k\cdot e_i)\nonumber
\end{eqnarray}
for $i=1,\dots,n$. Since $z^{\epsilon_k}$ is a viscosity solution of \eqref{zEq} and $\phi^{\epsilon_k,\delta,\rho}\in C^2((0,T)\times (0,\infty)^n\times\R^n)$,  we compute as in the introduction of this paper to get
\begin{align}\label{KeySupSolnIneq}
0&\le  -\phi^{\epsilon_k,\delta,\rho}_t -\frac{1}{2}\tr\left(d(p_k)\sigma\sigma^t d(p_k)\left(D^2_p \phi^{\epsilon_k,\delta,\rho} + \frac{1}{\epsilon_k}(D_p \phi^{\epsilon_k,\delta,\rho}-y_k)\otimes (D_p \phi^{\epsilon_k,\delta,\rho}-y_k) \right)\right) \nonumber\\
 &\le - \phi_t(t_k,p_k) - 
 \frac{1}{(1-\delta)^2}G(D^2u^\rho(x^{\epsilon_k,\delta}),Du^\rho(x^{\epsilon_k,\delta}),x^{\epsilon_k,\delta}; A_0)+ o(1). 
  \end{align}

\par Recall that from \eqref{lamdalpha}
\begin{equation}\label{KeySupSolnIneq2}
\lambda(A_0) +G(D^2u^\rho,Du^\rho,x; A_0)\le O(\rho^\alpha)
\end{equation}
fo all $x\in \R^n$. In particular, \eqref{KeySupSolnIneq2} and \eqref{KeySupSolnIneq} together imply
\begin{align*}
 0& \le-\phi_t(t_k,p_k) - \frac{1}{(1-\delta)^2}\lambda\left(A_0\right)+o(1)+O(\rho^\alpha)  \nonumber\\
&=-\phi_t(t_0,p_0) - \frac{1}{(1-\delta)^2}\lambda\left((1-\delta)^2d(p_0)D^2\phi(t_0,p_0)d(p_0)\right)+o(1)+O(\rho^\alpha) . \nonumber
\end{align*}
We obtain \eqref{psisuper} by letting $k\rightarrow \infty$ and then $ \delta,\rho\rightarrow 0^+$.

\par 4. We can argue analogously to conclude that $\overline{z}$ is subsolution of \eqref{SecPsiEq}; moreover, this argument is a bit easier than above as it turns one does not need to smooth the 
corrector function $u$. We leave the details to the interested reader.

 \par 5. In order to conclude the proof, we need to argue that $\overline{z}\le \underline{z}$.  Direct computation shows that the function 
 $$
z^{\eta}(t,p):=\overline{z}(t,p) - \eta\left(\frac{1}{t}+ \sum^n_{i=1}p_i\right)
 $$
 is a subsolution of \eqref{SecPsiEq} for each $\eta>0$. By Lemma \ref{TechnicalLemma} below, we have if $g$ satisfies \eqref{gEstimate} then
 $$
\varphi\le \underline{z}\le \overline{z}\le L,
 $$
 and if $g$ satisfies \eqref{gEstimate2} then
 \begin{equation}\label{IneqForgEst}
 \varphi\le \underline{z}\le \overline{z}\le L \sum^n_{i=1}p_i.
\end{equation}
Here $\varphi$ is the Black-Scholes price (that satisfies the PDE \eqref{OriginalBS}) and is given by 
$$
\varphi(p)=\int_{\R^n}g\left(p_1 e^{\sqrt{t}\sigma^te_1\cdot z - \frac{1}{2}|\sigma^t e_1|^2t},\dots, p_n e^{\sqrt{t}\sigma^te_n\cdot z- \frac{1}{2}|\sigma^t e_n|^2t}\right)\frac{e^{-|z|^2/2}}{(2\pi)^{n/2}}dz.
$$
When $g$ satisfies \eqref{gEstimate2}, the explicit formula above with inequality \eqref{IneqForgEst} gives
 $$
L= \lim_{|p|\rightarrow\infty}\frac{\varphi(p)}{\sum^n_{i=1}p_i}\le  \lim_{|p|\rightarrow\infty}\frac{\underline{z}(t,p)}{\sum^n_{i=1}p_i}\le  \lim_{|p|\rightarrow\infty}\frac{\overline{z}(t,p)}{\sum^n_{i=1}p_i}\le L.
 $$
Thus, when $g$ satisfies either  \eqref{gEstimate} or \eqref{gEstimate2},
$$
\lim_{|p|\rightarrow\infty}\frac{(z^\eta -\underline{z})(t,p)}{\sum^n_{i=1}p_i}=-\eta.
$$
It follows that $z^\eta -\underline{z}$ has a maximum at some $(t_0,p_0)\in (0,T]\times [0,\infty)^n.$

\par If $t_0=T$, then 
$$
\overline{z}\le \underline{z} + \eta\left(\frac{1}{T}+ \sum^n_{i=1}p_i\right).
$$ 
Letting $\eta\rightarrow 0^+$ leads to the desired inequality, $\overline{z}\le \underline{z}$. Now suppose $t_0<T$ and, for now, that $\overline{z},\underline{z}$ are smooth.  From 
calculus,
$$
\begin{cases}
z^\eta_t(t_0,p_0)=\underline{z}_t(t_0,p_0)\\
d(p_0)D^2z^\eta(t_0,p_0)d(p_0)\le d(p_0)D^2\underline{z}(t_0,p_0)d(p_0)
\end{cases}.
$$
However, these inequalities would imply a contradiction as
\begin{align*}
\frac{\eta}{t_0^2}&=-\overline{z}_t(t_0,p_0)+\underline{z}_t(t_0,p_0)
\le \lambda(d(p_0)D^2\overline{z}(t_0,p_0)d(p_0))-\lambda(d(p_0)D^2\underline{z}(t_0,p_0)d(p_0))\\
&= \lambda(d(p_0)D^2z^\eta(t_0,p_0)d(p_0))-\lambda(d(p_0)D^2\underline{z}(t_0,p_0)d(p_0))\le 0.
\end{align*}
The last inequality above is due to the monotonicity of $\lambda$.  It is now routine to use the ideas in Section 8 of  \cite{CIL} to make the same conclusion without assuming $\overline{z},\underline{z}$ are smooth.  
\end{proof}
\begin{lem}\label{TechnicalLemma} Let $z^\epsilon$ be the solution of \eqref{zEq} described in Proposition \ref{zSolnProp}. \\
(i) Then
\begin{equation}\label{zgrowth1}
\varphi(t,p)\le z^\epsilon(t,p,y)\le L + \sum^n_{i=1}\sqrt{\epsilon}p_i|y_i-L|, \quad (t,p,y)\in (0,T]\times (0,\infty)^n\times\R^n,
\end{equation}
provided $g$ satisfies \eqref{gEstimate} or 
\begin{equation}\label{zgrowth2}
\varphi(t,p)\le z^\epsilon(t,p,y)\le L \sum^n_{i=1}p_i+ \sum^n_{i=1}\sqrt{\epsilon}p_i|y_i-L|, \quad (t,p,y)\in (0,T]\times (0,\infty)^n\times\R^n,
\end{equation}
provided $g$ satisfies \eqref{gEstimate2}. Here $\varphi$ is the ``Black-Scholes" price
\begin{equation}\label{OriginalBS}
\begin{cases}
\varphi_t + \frac{1}{2}\tr\sigma\sigma^t\left(d(p)D^2\varphi d(p)\right)=0, \quad (t,p)\in (0,T)\times (0,\infty)^n\\
 \hspace{1.76in} \varphi = g, \quad (t,p)\in \{T\}\times (0,\infty)^n
\end{cases}.
\end{equation}
(ii) For each $0<\eta< T$, there is a $K=K(\eta)$ such that 
\begin{equation}\label{zgrowth3}
z^\epsilon(t,p,y)\ge \sum^n_{i=1}\sqrt{\epsilon}p_i|y_i| - K T\epsilon, \quad(t,p,y)\in (0,T-\eta]\times (0,\infty)^n\times\R^n
\end{equation}
for all $0<\epsilon<1/4.$
\end{lem}
We omit a proof as one is readily adapted from Proposition 2.1 and Lemma 2.2 in \cite{BS}. \\\\
{\bf Acknowledgements}: We thank Craig Evans and Yifeng Yu for their encouragement.

\appendix


\begin{thebibliography}{}

\bibitem{BC} Bardi, M.; Capuzzo-Dolcetta, I. \emph{Optimal control and viscosity solutions of Hamilton-Jacobi-Bellman equations}. With appendices by Maurizio Falcone and Pierpaolo Soravia. Systems \& Control: Foundations \& Applications. BirkhŠuser Boston, Inc., Boston, MA, 1997. 

\bibitem{BS} Barles, G. ; Soner, H.  \emph{Option pricing with transaction costs and a nonlinear Black-Scholes equation}. Finance and Stochastics. 2, 369-397 (1998).

\bibitem{BP} Barles, G. ; Perthame, B.  \emph{Discontinuous solutions of deterministic optimal stopping time problems}.  Mod\`{e}l. Math. et Anal. Number 21 (1987)  pp. 557--579.

\bibitem{RC} Carmona, R. \emph{Indifference Pricing: Theory and Applications}. Princeton University Press, 2009.

\bibitem{CIL} Crandall, M. G.;Ishii, H.;Lions, P.-L. \emph{User's guide to viscosity solutions of second order partial differential equations}. Bull. Amer. Math. Soc. (N.S.) 27 (1992), no. 1, 1--67.

\bibitem{CI} Crandall, Michael G.; Ishii, Hitoshi. \emph{The maximum principle for semicontinuous functions}. Differential Integral Equations 3 (1990), no. 6, 1001--1014.

\bibitem{DPZ} Davis, M.; Panas, V.; Zariphopoulou, T. \emph{European option pricing under transaction fees}. SIAM J. Cont. Opt. 31, 470-493 (1993).

\bibitem{Evans} Evans, L. C. \emph{Partial differential equations}. Graduate Studies in Mathematics, 19. American Mathematical Society, Providence, RI, 1998.

\bibitem{LCE} Evans, L. C. \emph{The perturbed test function technique for viscosity solutions of partial differential equations}. Proc. Royal Soc. Edinburgh Sect. A. 111, 359-375 (1989).

\bibitem{FS} Fleming, W.; Soner, H. \emph{Controlled Markov processes and viscosity solutions}. Second edition. Stochastic Modeling and Applied Probability, 25. Springer, New York, 2006.

\bibitem{GT} Gilbarg, D.; Trudinger, N. \emph{Elliptic Partial Differential Equations of Second Order}. Springer (1998). 

\bibitem{Hynd1} Hynd, R.  \emph{Analysis of Hamilton-Jacobi-Bellman equations arising in stochastic singular control}. ESAIM Control Optim. Calc. Var. 19 (2013), no. 1, 112--128.

\bibitem{Hynd2} Hynd, R.  \emph{The eigenvalue problem of singular ergodic control}. Communications of Pure and Applied Mathematics, 65 (2012), no. 5, 649--682. 

\bibitem{KO}  Korevaar, N. J. \emph{Convex solutions to nonlinear elliptic and parabolic boundary value problems}. Indiana Univ. Math. J. 32 (1983), no. 4, 603--614.

\bibitem{LPV} Lions, P.-L.; Papanicolaou, G.; Varadhan, S.  \emph{Homogenization of Hamilton--Jacobi equations}. Unpublished, circa 1988.




\end{thebibliography}
\end{document}